\documentclass[11pt]{amsart}
\usepackage{epsfig}
\usepackage{amsmath}
\usepackage{amsthm}
\usepackage{amssymb}
\DeclareMathOperator{\aaa}{\bf{a}}
\DeclareMathOperator{\bbb}{\bf{b}}
\DeclareMathOperator{\comp}{c}
\DeclareMathOperator{\inv}{inv}
\DeclareMathOperator{\Kappa}{K}
\DeclareMathOperator{\sgn}{sgn}
\DeclareMathOperator{\Vol}{Vol}

\newcommand{\covered}{<\hspace{-4pt}\cdot\hspace{2pt}}
\newcommand{\covers}{\hspace{2pt}\cdot\hspace{-4pt}>}
\newcommand{\va}{\vec{a}}
\newcommand{\vb}{\vec{b}}

\newtheorem{theorem}{Theorem}[section]
\newtheorem{lemma}[theorem]{Lemma}

\theoremstyle{definition}

\begin{document}

\title[Generalized parking functions and descent numbers]
{Generalized parking functions, descent numbers,
and chain polytopes of ribbon posets}

\author{Denis Chebikin and Alexander Postnikov}


\address{Department of Mathematics, Massachusetts Institute of
  Technology,
Cambridge, MA 02139}

\email{chebikin@gmail.com, apost@math.mit.edu}

\begin{abstract}
We consider the inversion enumerator $I_n(q)$, which
counts labeled trees or, equivalently, parking functions.
This polynomial has a natural extension to generalized parking functions.
Substituting $q=-1$ into this generalized polynomial produces the number
of permutations with a certain descent set. In the classical case,
this result implies the formula $I_n(-1) = E_n$, the number of
alternating permutations.
We give a combinatorial proof of these formulas based on the
involution principle.
We also give a geometric interpretation of these identities in terms
of volumes of
generalized chain polytopes of ribbon posets.  The volume of such a
polytope is given by
a sum over generalized parking functions, which is similar to an expression
for the volume of the parking function polytope of Pitman and Stanley.
%
%We consider the inversion enumerator $I_n(q)$, which
%can be
%defined in terms of either labeled trees or parking functions,
%and the natural extension to generalized parking functions.
%We show that substituting $q=-1$ into this polynomial produces the number
%of permutations with a given descent set. In the classical case,
%this result implies the formula $I_n(-1) = E_n$, the number of
%alternating permutations. We give a new combinatorial proof of this formula.
%We also consider a family of polytopes generalizing the chain polytopes
%of ribbon posets in that the upper bound of $1$ on the sum of coordinates
%corresponding to a chain of the poset is replaced with an arbitrary number
%$d_i$. The volume of such a polytope is 
%a polynomial in the $d_i$'s
%given by a sum over generalized
%parking functions which is very similar to an expression for the volume
%of the parking function polytope of Pitman and Stanley.
%Setting $d_i=1$ yields a geometric interpretation of the aforementioned
%identities involving descent numbers.
\end{abstract}

\maketitle 

\section{introduction}\label{intro}

Let $\mathcal{T}_n$ be the set of all trees on vertices labeled
$0$, $1$, $2$, \dots, $n$ rooted at $0$. For $T\in\mathcal{T}_n$, let
$\inv(T)$ be the number of pairs $i>j$ such that $j$ is a descendant of $i$
in $T$. Define the $n$-th \emph{inversion enumerator} to be the polynomial
$$
I_n(q):=\sum_{T\in\mathcal{T}_n} q^{\inv(T)}.
$$
Another way to define this polynomial
is via parking functions. A sequence $(b_1,b_2,\dots,b_n)$ of
positive integers is a \emph{parking function of length $n$}
if for all $1\leq j \leq n$,
at least $j$ of the $b_i$'s do not exceed $j$. A classical bijection of
Kreweras \cite{Kreweras}
establishes a correspondence between trees in $\mathcal{T}_n$ with
$k$ inversions and parking functions of length $n$ whose components add up to
${n+1\choose 2} - k$. Hence we can write
$$
I_n(q) = \sum_{(b_1,\dots,b_n)\in\mathcal{P}_n}
q^{{n+1\choose 2} - b_1-b_2-\dots-b_n},
$$
or
$$
\sum_{(b_1,\dots,b_n)\in\mathcal{P}_n} q^{b_1+b_2+\dots+b_n-n}=
q^{n\choose 2}\cdot I_n(q^{-1}),
$$
where $\mathcal{P}_n$ is the set of all parking functions of length $n$.
Cayley's formula states that $|\mathcal{T}_n|=|\mathcal{P}_n|=(n+1)^{n-1}$,
hence $I_n(1) = (n+1)^{n-1}$.

Here we focus on the formula
\begin{equation}\label{In_minus_one}
I_n(-1) = E_n,
\end{equation}
where $E_n$ is the
$n$-th \emph{Euler number}, most commonly defined as the
number of permutations $\sigma_1\sigma_2\dots\sigma_n$ of $[n]=
\{1,2,\dots,n\}$ such that $\sigma_1 < \sigma_2 > \sigma_3 < \sigma_4 >\dots$,
called \emph{alternating} permutations. This formula can be obtained
by deriving 
a closed form expression for
the generating function
$\sum_{n\geq 0} I_n(q) x^n/n!$ and showing that setting $q=-1$ yields
$\tan x + \sec x = \sum_{n\geq 0} E_n x^n/n!$
(see the paper \cite{Gessel} by Gessel or Exercises
3.3.48--49 in \cite{GouldenJackson}). 
A~direct combinatorial proof was given
by
Pansiot \cite{Pansiot}.
In this paper we give two
other ways to prove this fact, one of which,
presented in Section \ref{parking_functions},
is an involution argument on the set of all
but $E_n$ members of $\mathcal{P}_n$. This involution
is a special case of a more general argument
valid for a broader version of parking functions, which we now describe.

%Let $\gamma = (\gamma_1,\gamma_2,\dots,\gamma_k)$ be a 
%\emph{weak composition} of $n$,
%that is, $\gamma_i$ are non-negative integers adding up to $n$. 
Let $\va=(a_1,a_2,\dots,a_n)$ be a \emph{non-decreasing} sequence
of positive integers.
Let us call
a sequence $(b_1,b_2,\dots,b_n)$ of positive integers an
\emph{$\va$-parking function} if the increasing rearrangement
$b_1' \leq b_2' \leq \dots \leq b_n'$ of this sequence satisfies
$b_i' \leq a_i$ for all $i$. Note that
$(1,2,\ldots,n)$-parking functions
are the regular parking functions of length $n$. 
These $\va$-parking functions are
$(a_1, a_2 - a_1, a_3-a_2,\dots)$-parking functions in the original notation
of Yan \cite{Yan}, but the present definition is consistent with later
literature, such as the paper \cite{KungYan} of Kung and Yan.
Let $\mathcal{P}_{\va}$
be the set of all $\va$-parking functions, and define
\begin{equation}\label{Igamma}
I_{\va}(q) := \sum_{(b_1,\dots,b_n)\in\mathcal{P}_{\va}}
q^{b_1+b_2+\dots+b_n-n}
\end{equation}
(this is the \emph{sum enumerator} studied in \cite{KungYan}).
For a subset $S\subseteq [n-1]$, let $\beta_n(S)$ be the number of permutations
of size $n$ with descent set $S$. 
In Section~\ref{parking_functions}
(Theorem~\ref{Igamma_minus_one_thm}) we prove the following generalization
of (\ref{In_minus_one}):
\begin{equation}\label{Igamma_minus_one}
|I_{\va}(-1)| = \left\{
\begin{array}{ll}
0,& \mbox{if $a_1$ is even;}\\
\beta_n(S),& \mbox{if $a_1$ is odd,}
\end{array}
\right.
\end{equation}
where
\begin{equation}\label{Igamma_S}
S = \Bigl\{i\in [n-1]\ \Bigl|\ a_{i+1}\mbox{ is odd}\Bigr\}.
\end{equation}
Indeed, for $\va = (1,2,\ldots,n)$ we have
$S = \{2,4,6,\dots\}\cap [n-1]$, so that $\beta_n(S)$ counts alternating
permutations of size $n$. The formula (\ref{Igamma_minus_one}) arises in
a more sophisticated algebraic context in the paper \cite{PakPostnikov}
of Pak and Postnikov.

In Section \ref{chain_polytopes} 
we obtain a geometric
interpretation of these results by considering generalized
chain polytopes of ribbon posets. 
Given a subset $S \subseteq \{2,3,\dots,n-1\}$,
define $u_S = u_1u_2\dots u_{n-1}$ to be the monomial in non-commuting
formal variables $\aaa$ and $\bbb$ with $u_i = \aaa$ if $i\notin S$ and
$u_i = \bbb$ if $i\in S$. Let $\comp(S)$ be the
%\emph{ordinary}
composition
$(1,\delta_1,\delta_2,\dots,\delta_{k-1})$ of $n$,
%(that is, $\delta_i>0$ for all $i$),
where the $\delta_i$'s are defined by
$u_S = \aaa^{\delta_1} \bbb^{\delta_2} \aaa^{\delta_3} \bbb^{\delta_4}
\dots$.
For example, for $n=7$ and $S = \{2,3,4\}$ we have
$u_S = \aaa\bbb\bbb\bbb\aaa\aaa = \aaa\bbb^3\aaa^2$, so
$\comp(S) = (1,1,3,2)$. 
Now define the polytope $\mathcal{Z}_S(d_1,d_2,\dots,
d_k)$, where $0<d_1\leq d_2 \leq \cdots \leq d_k$ are real numbers,
to be the set of
all points $(x_1,x_2,\dots,x_n)$ satisfying
the inequalities $x_j \geq 0$ for $j\in [n]$,
$x_1 \leq d_1$, and
$$
x_{\delta_1+\delta_2+\dots+\delta_{i-1}+1}
+ x_{\delta_1+\delta_2+\dots+\delta_{i-1}+2}
+ \dots + x_{\delta_1+\delta_2+\dots+\delta_i + 1} \leq d_{i+1}
$$
for $1\leq i \leq k-1$. Thus to the above example corresponds the polytope
$\mathcal{Z}_S(d_1,d_2,d_3,d_4)$ in $\mathbb{R}_{\geq 0}^7$ defined by
$$
x_1 \leq d_1;
$$
$$
x_1+x_2 \leq d_2;
$$
$$
x_2+x_3+x_4+x_5 \leq d_3;
$$
$$
x_5+x_6+x_7 \leq d_4.
$$
We require $1\notin S$ here to ensure that $\delta_1\neq 0$, but
there is no essential loss of generality because the chain polytope
of the poset $Z_S$ is defined by the same relations as $Z_{[n-1]-S}$.

For a poset $P$ on $n$ elements, the
\emph{chain polytope} $\mathcal{C}(P)$ is the set of points
$(x_1,x_2,\dots,x_n)$ of the unit hypercube $[0,1]^n$ satisfying the
inequalities $x_{p_1} + x_{p_2} + \dots + x_{p_\ell} \leq 1$ for every
chain $p_1 < p_2 < \dots < p_\ell$ in $P$; see \cite{StanleyPosetPolytopes}.
Hence
$\mathcal{Z}_S(1,1,\dots)$ is the chain polytope of the
\emph{ribbon poset} $Z_S$, which is the poset on $\{z_1,z_2,\dots,z_n\}$
generated by the cover relations $z_i \covers z_{i+1}$ if $i\in S$ and
$z_i \covered z_{i+1}$ if $i\notin S$. The volume of $\mathcal{C}(P)$ equals
$1/n!$ times the number of linear extensions of $P$, which in the case
$P = Z_S$ naturally correspond to permutations of size $n$ with descent set
$S$. Our main result concerning the polytope $\mathcal{Z}_S$ is a formula
for its volume. For a composition $\gamma$ of $n$, 
let $\Kappa_\gamma$ denote the set of \emph{weak} compositions $\alpha
= (\alpha_1, \alpha_2, \dots)$ of $n$,
meaning that $\alpha$ can have parts equal to $0$, such that
$\alpha_1  + \alpha_2 +\dots + \alpha_i \geq \gamma_1 + \gamma_2 + \dots
+ \gamma_i$ for all $i$.
Define $\va(\gamma)$ to be the sequence consisting of
$\gamma_1$ $1$'s, followed by $\gamma_2$ $2$'s, followed by
$\gamma_3$ $3$'s, and so on.
Then $\alpha$ is in $\Kappa_\gamma$ if and only
if
$\alpha$ is the content of an $\va(\gamma)$-parking function.
(The \emph{content} of a parking function
is the composition whose $i$-th part is the number of
components of the parking function equal to $i$.)
In Section~\ref{chain_polytopes}
(Theorem~\ref{vol_ZS_thm}) we show that
\begin{equation}\label{vol_ZS}
n!\cdot \Vol\bigl(\mathcal{Z}_S(d_1,d_2,\dots,d_k)\bigr)=
\left|
\sum_{(b_1,\dots,b_n)\in\mathcal{P}_{\va(\comp(S))}} \prod_{i=1}^n
(-1)^{b_i} d_{b_i}
\right|
=
\end{equation}
$$
=
\left|
\sum_{\alpha \in \Kappa_{\comp(S)}} 
{n\choose \alpha}\cdot (-1)^{\alpha_1+\alpha_3+\alpha_5+\dots}\cdot
d_1^{\alpha_1} d_2^{\alpha_2} \dots d_k^{\alpha_k}
\right|,
$$
where ${n\choose \alpha} = {n!\over \alpha_1!\alpha_2!\dots\alpha_k!}$
and $k = \ell(\comp(S))$ is the number of parts of
$\comp(S)$. For
example, for $n=5$ and $S=\{4\}$, 
we have
$$
\Kappa_{\comp(S)} = \Kappa_{(1,3,1)} = 
\{(1,3,1),(1,4,0),(2,2,1),(2,3,0),(3,1,1),(3,2,0),
$$
$$
(4,0,1), (4,1,0), (5,0,0)\},
$$
so
we get from (\ref{vol_ZS}) that
$$
5!\cdot\Vol\bigl(\mathcal{Z}_{\{1\}}(d_1,d_2,d_3)\bigr)
=  20 d_1   d_2^3 d_3
-  5  d_1   d_2^4
-  30 d_1^2 d_2^2 d_3
+  10 d_1^2 d_2^3
$$
$$
+ 20  d_1^3 d_2   d_3
- 10  d_1^3 d_2^2
- 5   d_1^4       d_3
+ 5   d_1^4 d_2
-     d_1^5.
$$

Setting $d_i = q^{i-1}$ in (\ref{vol_ZS}), where we take
$q\geq 1$ so that the sequence $d_1$, $d_2$, \ldots\ is non-decreasing,
and recalling~(\ref{Igamma}) gives
$$
n!\cdot \Vol(1,q,q^2,\dots) =
\left|
%(-1)^n
\sum_{(b_1,\ldots,b_n)\in\mathcal{P}_{\va(\comp(S))}}
(-q)^{b_1+b_2+\cdots+b_n-n}
\right|
= \left|
I_{\va(\comp(S))}(-q)
\right|.
$$
Specializing further by setting $q=1$ yields the identity
$$
\left|I_{\va(\comp(S))}(-1)\right| = \beta_n(S).
$$
Observe that this identity is consistent with (\ref{Igamma_minus_one}).
Indeed, the first part of $\comp(S)$ is positive, and thus the first
element of $\va(\comp(S))=(a_1,a_2,\dots,a_n)$
is $1$, i.e.\ an odd number. Comparing the sequence $(a_1,a_2,\dots,a_n)$ with
the letters of the word $\bbb u_S$ we see that $a_{i+1}
= a_i+1$ if the corresponding letters of $u_S$ are different, and
$a_{i+1}=a_i$ otherwise; in other words, changes of parity between
consecutive elements of $(a_1,a_2,\dots,a_n)$ correspond to letter changes
in the word $\bbb u_S$. (The extra $\bbb$ in front corresponds to the first
part $1$ of $\comp(S)$.) 
For example, for $n=7$ and $S=\{2,3,4\}$, we have
$\comp(S) = (1,1,3,2)$, $\bbb u_S = \bbb\aaa\bbb^3\aaa^2$, and
$\va(\comp(S)) = (1,2,3,3,3,4,4)$.
It follows that the subset constructed from
$\va(\comp(S))$ according to the rule (\ref{Igamma_S}) of an earlier result
is $S$, so
the results agree.

Considering once more the case $S=\{2,4,6,\dots\}\cap[n-1]$, let us point
out the similarity between the formula (\ref{vol_ZS}) and the expression
that Pitman and Stanley \cite{PitmanStanley} derive for the volume of their
\emph{parking function polytope}. This polytope,
which we denote by $\Pi_n(c_1,c_2,\dots,c_n)$,
is defined by the inequalities $x_i \geq 0$ and
$$
x_1+x_2+\dots + x_i \leq c_1+c_2+\dots+c_i
$$
for all $i\in [n]$. The volume-preserving change of coordinates
$y_i = c_n+c_{n-1} + \dots + c_{n+1-i} - (x_1+x_2+\dots+x_i)$ transforms the
defining relations above into $y_i \geq 0$ for $i\in [n]$, $y_1 \leq c_1$, and
$y_i - y_{i+1} \geq c_i$ for $i\in [n-1]$, and these new relations look much
like the ones defining $\mathcal{Z}_{\{2,4,6,\dots\}}(c_1,c_2,\dots,c_n)$:
in essence we have here a difference instead of a sum. This similarity 
somewhat explains the close resemblance of the volume formulas for the two
polytopes,
as for $\Pi_n(c_1,c_2,\dots,c_n)$ we have
$$
n!\cdot \Vol(\Pi_n(c_1,c_2,\dots,c_n))
= \sum_{(b_1,\dots,b_n)\in \mathcal{P}_n} \prod_{i=1}^n c_{b_i} 
= \sum_{\alpha \in \Kappa_{1^n}} \prod_{i=1}^n {n\choose \alpha}c_i^{\alpha_i}.
$$

\section*{Acknowledgments}

We would like to thank Ira Gessel for providing useful references.

\section{An involution on $\va$-parking functions}\label{parking_functions}

The idea of the
combinatorial argument presented in this section was first discovered by
the second author and Igor Pak during their work on \cite{PakPostnikov}.

Let $\va=(a_1,a_2,
\dots,a_n)$ be a non-decreasing sequence of positive integers.
As a first step in the construction of our involution on
$\va$-parking functions, let $Y_{\va}$ be the Young diagram whose
column lengths from left to right are $a_n$, $a_{n-1}$, \dots, $a_1$.
Define a \emph{horizontal strip} $H$ inside $Y_{\va}$ to be a set of cells
of $Y_{\va}$ satisfying the following conditions:

(i) for every $i\in [n]$, the set
$H$ contains exactly one cell $\sigma_i$ from column $i$
(we number the columns $1$, $2$, \dots, $n$
from left to right);

(ii) for $i < j$, the cell $\sigma_i$ is in the same or in a lower row than
the cell $\sigma_j$.

For a horizontal strip $H$, let us call a filling of the cells of $H$ with
numbers $1$, $2$, \dots, $n$ \emph{proper} if the numbers in row $i$ are in
increasing order if $i$ is odd, or in decreasing order if $i$ is even (we
number the rows $1$, $2$, \dots, from top to bottom). Let $\mathcal{H}_{\va}$
denote the set of all properly filled horizontal strips inside $Y_{\va}$.

For an $\va$-parking function $\vb = (b_1,b_2,\dots,b_n)$, define
$H(\vb)\in\mathcal{H}_{\va}$ in the following way.
Let $I_j \subseteq [n]$ be the set
of indices $i$ such that $b_i=j$. Construct
the filled horizontal strip $H(\vb)$ by first writing the elements
of $I_1$ in \emph{increasing} order in the $|I_1|$ \emph{rightmost} columns 
in row $1$ of $Y_{\va}$, then writing the elements of $I_2$ in
\emph{decreasing} order in the next $|I_2|$ columns from the right in
row $2$,
then writing the elements of $I_3$ in \emph{increasing} order
in the next $|I_3|$ columns from the right in row $3$, and so on, alternating
between increasing and decreasing order.

\begin{lemma}\label{Hb_fits_into_Y}
A sequence
$\vb$ is an $\va$-parking function if and only if the horizontal strip
$H(\vb)$ produced in the above construction fits into $Y_{\va}$.
\end{lemma}

\begin{proof}
Let $(b_1',b_2',\ldots,b_n')$ be the increasing rearrangement of $\vb$. Then
the cell of $H(\vb)$ in the $i$-th column from the right belongs to row
$b_i'$. Thus the condition of the lemma is equivalent to $b_i' \leq a_i$ for
all $i$. 
\end{proof}

Clearly, the filling of $H(\vb)$ from the above procedure is proper, and
hence $\vb \leftrightarrow H(\vb)$ is a bijection between $\mathcal{P}_{\va}$
and $\mathcal{H}_{\va}$. We will describe our involution on
$\va$-parking functions in terms of the corresponding filled
horizontal strips.

Let $H$ be a properly filled horizontal strip in $\mathcal{H}_{\va}$.
In what follows we use $\sigma_i$ to refer
to both the cell of $H$ in column $i$ and to the number written in it.
Let $r(\sigma_i)$ be the number of the row containing $\sigma_i$.
We begin by defining the \emph{assigned direction} for each of
the $\sigma_i$'s. 
For the purpose of this definition it is convenient to imagine that
$H$ contains a cell labeled $\sigma_{n+1}=n+1$ in row $1$ and column $n+1$,
that is, just outside the first row $Y_{\va}$ on the right.
%We say that the assigned direction for $\sigma_i$
%is \emph{down} if either of the two conditions holds:
Let
$$
\epsilon_i = \sgn(\sigma_i-\sigma_{i+1})(-1)^{r(\sigma_i)},
$$
where $\sgn(x)$ equals $1$ if $x>0$, or $-1$ if $x<0$.
Define the assigned direction for $\sigma_i$ to be
\emph{up} if $\epsilon_i = -1$ and \emph{down} if $\epsilon_i=1$.

Let us say that
$\sigma_i$ is \emph{moveable down} if the assigned direction for $\sigma_i$
is down, $\sigma_i$ is not the bottom cell of column $i$, and moving 
$\sigma_i$ to the cell immediately below it would not violate the rules of
a properly filled horizontal strip. The latter condition prohibits moving
$\sigma_i$ down if there is another cell of $H$ immediately to the left of it,
or if moving $\sigma_i$ down
by one row would violate the rule for the relative order
of the numbers in row $r(\sigma_i)+1$.
Let us say that $\sigma_i$ is \emph{moveable up} if the assigned direction
for $\sigma_i$ is up. Note that we do not need any complicated conditions
in this case: if $\sigma_i$ has another cell of $H$ immediately to its right,
or if $\sigma_i$ is in the top row, then the assigned direction for
$\sigma_i$ is down.

It is a good time to consider an example. Figure \ref{hor_strip_example}
shows the diagram $Y_{\va}$ and a 
properly filled horizontal strip $H\in\mathcal{H}_{\va}$ for
$\va=(3,3,6,7,7,7,8)$.
The horizontal strip shown is $H(\vb)$, where $\vb = (5,7,2,5,1,5,2) \in
\mathcal{P}_{\va}$. Moveable cells are equipped with arrows pointing in
their assigned directions. Note that the assigned direction for
$\sigma_7 = 5$ is down
because of the ``imaginary''
$\sigma_8=8$ next to it, but it is not moveable down because the
numbers $7$, $3$, $5$ in row $2$ would not be ordered properly. The assigned
direction for $\sigma_3=4$ and $\sigma_4=6$ is also down, but these cells
are not moveable down because moving them down would not produce a horizontal
strip.

\begin{figure}[h!]
\begin{center}
\input{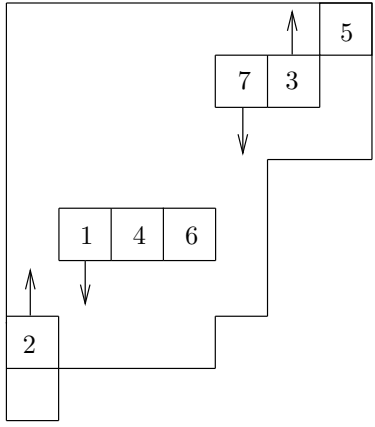}
\end{center}
\caption{A properly filled horizontal strip with assigned directions for
its cells}
\label{hor_strip_example}
\end{figure}

The validity of the involution we are about to define depends on the following
simple but crucial fact.

\begin{lemma}\label{move_back_forth}
Let $\sigma_i$ be a moveable cell of $H\in\mathcal{H}_{\va}$, and let
$H'\in\mathcal{H}_{\va}$ be the horizontal strip obtained from $H$ by
moving $\sigma_i$ by one row in its assigned direction.
Then

(a) $\sigma_i$
is moveable in the opposite direction in $H'$;

(b) if $j \neq i$ and $\sigma_j$ is moveable in $H$, then $\sigma_j$
is moveable in
the same direction in $H'$;

(c) if $\sigma_j$ is not moveable in $H$, then it is not
moveable in $H'$.
\end{lemma}

\begin{proof}
Follows by inspection of the moving rules.
\end{proof}

Let $\tilde{\mathcal{H}}_{\va}$
be the set of all $H\in\mathcal{H}_{\va}$ such that
$H$ has at least one moveable cell (up or down). Define the map
$\psi : \tilde{\mathcal{H}}_{\va} \rightarrow \tilde{\mathcal{H}}_{\va}$
as follows:
given $H\in\tilde{\mathcal{H}}_{\va}$, let $\psi(H)$ be the horizontal
strip obtained from $H$ by choosing the \emph{rightmost} moveable cell
of $H$ and moving it by one row in its assigned direction. In view of
Lemma \ref{move_back_forth}, $\psi$ is an involution.

For $\vb = (b_1,b_2,\dots,b_n)\in \mathcal{P}_{\va}$
and $H = H(\vb)$, define
$s(\vb) = s(H) := b_1 + b_2+\dots+b_n - n$. Observe that $s(\vb)$ is the number
of cells of $Y_{\va}$ that lie above one of the cells of $H(\vb)$. In the
example in Figure \ref{hor_strip_example}, we have $s(\vb)=s(H)=
6+4+4+4+1+1+0=20$. Clearly, $s(H) = s(\psi(H)) \pm 1$ for $H\in
\tilde{\mathcal{H}}_{\va}$. Since $\psi$ is fixed-point free,
it follows that
$$
\sum_{H\in\tilde{\mathcal{H}}_{\va}} (-1)^{s(H)} = 0,
$$
and that
\begin{equation}\label{Igamma_minus_one_remainder}
I_{\va}(-1) = \sum_{H\in\tilde{\mathcal{H}}_{\va}-\mathcal{H}_{\va}}
(-1)^{s(H)}
\end{equation}
(cf.\ (\ref{Igamma})). It remains to examine the members of
$\mathcal{H}_{\va}-\tilde{\mathcal{H}}_{\va}$ in order to evaluate the
right hand side of (\ref{Igamma_minus_one_remainder}).

\begin{lemma}\label{fixed_points}
If $a_1$ is even, then $\mathcal{H}_{\va}-\tilde{\mathcal{H}}_{\va}
= \varnothing$. If $a_1$ is odd, then 
$\mathcal{H}_{\va}-\tilde{\mathcal{H}}_{\va}$ consists of all filled
horizontal strips $H$ in $Y_{\va}$ such that 
the cell $\sigma_i$ of $H$ is at the bottom of column $i$ for all
$i\in [n]$, and
the permutation
$\sigma_n \dots \sigma_2\sigma_1$ has descent set
$$
S = \Bigl\{i\in [n-1]\ \Bigl|\ a_{i+1}\mbox{ is odd}\Bigr\}
$$
(cf.\ (\ref{Igamma_S})).
\end{lemma}

\begin{proof}
Let $H\in\mathcal{H}_{\va}$, and consider the cell
$\sigma_n$ in $H$.
Since $\sigma_n < \sigma_{n+1}=n+1$, it follows that $\sigma_n$ is moveable
up if $r(\sigma_n)$ is even, or moveable down if $r(\sigma_n)$ is odd, unless
in the latter case $\sigma_n$ is at the bottom of column $n$. Thus if
$a_1$ is even, that is, if the rightmost column of $Y_{\va}$ has even
height, $\sigma_n$ is always moveable and
$\mathcal{H}_{\va}=\tilde{\mathcal{H}}_{\va}$.

Suppose that $a_1$ is odd, and let
$H\in\mathcal{H}_{\va}-\tilde{\mathcal{H}}_{\va}$.
Then no cells of $H$ are moveable, and hence the assigned direction for
every cell is down.

First, let us show that all cells of $H$ are at the bottom of their respective
columns. Suppose it is not so, and choose the \emph{leftmost}
cell $\sigma_i$ of $H$ such that
the cell immediately below it is in $Y_{\va}$. Our choice guarantees that
$\sigma_i$ does not have another cell of $H$ immediately to its left, so the
only way $\sigma_i$ can be not moveable down is if $\sigma_{i-1}$ is in
column $i-1$ one row below $\sigma_i$ and $\sgn(\sigma_{i-1}-\sigma_i)
= -(-1)^{r(\sigma_{i-1})}$.
% = (-1)^{r(\sigma_i)}
But in this case
the assigned direction for $\sigma_{i-1}$ is up --- a contradiction.

Now let us compute the descent set of $\sigma_n\sigma_{n-1}\dots \sigma_1$.
We just proved that $r(\sigma_{n+1-i}) = a_{i}$ for all $i\in [n]$.
For $i\in [n-1]$, we have
$$
1 = \epsilon_{n-i} = 
\sgn(\sigma_{n-i} - \sigma_{n+1-i})(-1)^{r(\sigma_{n-i})},
$$
and hence $\sigma_{n+1-i} > \sigma_{n-i}$, that is, 
$\sigma_n\sigma_{n-1}\cdots\sigma_1$ has a descent in position $i$,
if and only if $r(\sigma_{n-i}) = a_{i+1}$ is odd. 
\end{proof}

Note that from Lemma \ref{fixed_points} it follows that for all
$\mathcal{H}_{\va}-\tilde{\mathcal{H}}_{\va}$, the value
of $s(H)$ is the same, namely,
$a_1 + a_2 + \cdots + a_n - n$.
Combining with (\ref{Igamma_minus_one_remainder}),
we obtain the following theorem (cf.\ (\ref{Igamma_minus_one})).

\begin{theorem}[cf.\ \cite{PakPostnikov}]\label{Igamma_minus_one_thm}
For a non-decreasing
sequence $\va = (a_1,a_2,\dots,a_n)$ of positive integers, we have
$$
I_{\va}(-1) = \left\{
\begin{array}{ll}
0,& \mbox{if $a_1$ is even;}\\
(-1)^{a_1+\cdots+a_n-n}\cdot\beta_n(S),& \mbox{if $a_1$ is odd,}
\end{array}
\right.
$$
where $S = \Bigl\{i\in [n-1]\ \Bigl|\ a_{i+1}\mbox{ is odd}\Bigr\}$,
and $\beta_n(S)$ is the number of permutations of $[n]$ with descent set $S$.
\end{theorem}

\section{Generalized chain polytopes of ribbon posets}
\label{chain_polytopes}

In this section we prove the formula (\ref{vol_ZS}) of Section \ref{intro}:

\begin{theorem}\label{vol_ZS_thm}
For a positive integer $n$, a subset $S\subseteq [n-1]$ such that
$$\comp(S) = (1,\delta_1, \delta_2, \ldots, \delta_{k-1}),$$
and a sequence $0 < d_1 \leq d_2 \leq \cdots \leq d_k$ of real numbers,
we have
\begin{eqnarray}\label{vol_ZS_2}
\nonumber
n!\cdot \Vol\bigl(\mathcal{Z}_S(d_1,d_2,\dots,d_k)\bigr)&=&
(-1)^{1+\delta_2+\delta_4+\cdots}
\sum_{(b_1,\dots,b_n)\in\mathcal{P}_{\va(\comp(S))}} \prod_{i=1}^n
(-1)^{b_i} d_{b_i}
\\
&=&
(-1)^{1+\delta_2+\delta_4+\cdots}
\sum_{\alpha \in \Kappa_{\comp(S)}} 
{n\choose \alpha}\cdot (-1)^{\alpha_1+\alpha_3+\alpha_5+\dots}\cdot
d_1^{\alpha_1} 
%d_2^{\alpha_2}
\cdots d_k^{\alpha_k}.
\end{eqnarray}
\end{theorem}

\begin{proof}
First, note that the
expression in the right hand side of (\ref{vol_ZS_2}) is obtained
from the middle one by grouping together the terms corresponding
to all ${n\choose \alpha}$ $\va$-parking functions of content $\alpha$;
each of these terms equals 
$$
(-1)^{\alpha_1+2\alpha_2+3\alpha_3+\cdots}\cdot d_1^{\alpha_1}
\cdots d_k^{\alpha_k}
= (-1)^{\alpha_1+\alpha_3+\alpha_5+\cdots}\cdot d_1^{\alpha_1}
\cdots d_k^{\alpha_k}.
$$

In what follows we prove the equality between the left and the right hand sides of
(\ref{vol_ZS_2}).
For $i\in [k]$, let $\rho_i = 1+\delta_1+\delta_2+\cdots+\delta_{i-1}$.
The volume of $\mathcal{Z}_S(d_1,d_2,\ldots,d_k)$ can be expressed as
the following iterated integral:
\begin{equation}\label{integrals}
\int_0^{d_1}
\int_0^{d_2 - x_{1}}
\int_0^{d_2 - x_{1} - x_{2}}
\dots
\int_0^{d_2 - x_{1} - x_{2} - \cdots - x_{\rho_2-1}}
\end{equation}
$$
\int_0^{d_{3}-x_{\rho_{2}}}
\int_0^{d_{3}-x_{\rho_{2}}-x_{\rho_{2} + 1}}
\dots
\int_0^{d_{3}-x_{\rho_{2}}-x_{\rho_{2} + 1}-\cdots-x_{\rho_{3}-1}}
$$
$$
\dots
$$
$$
\int_0^{d_{k}-x_{\rho_{k-1}}}
\int_0^{d_{k}-x_{\rho_{k-1}}-x_{\rho_{k-1}+1}}
\dots
\int_0^{d_{k}-x_{\rho_{k-1}}-x_{\rho_{k-1}+1}-\cdots-x_{\rho_k-1}}
dx_n\ dx_{n-1}\ \cdots\ dx_1
$$
(Similar integral formulas appear in \cite{KungYan}
and in \cite[Sec.\ 18]{PostnikovStanley}.)
Note that the assumption $d_1\leq d_2 \leq \cdots \leq d_k$ validates
the upper limits of those integrals taken with respect to
variables $x_2$,
$x_{\rho_2+1}$, $x_{\rho_3+1}$, \dots, $x_{\rho_{k-1}+1}$:
for $2\leq i\leq k-1$, the condition
$$
x_{\rho_i+1} \leq d_{i} - x_{\rho_{i-1}+1} - x_{\rho_{i-1}+2} -
\cdots - x_{\rho_i}
$$
implies that
$$
d_{i+1} - x_{\rho_i+1} \geq 0,
$$
and $x_1 \leq d_1$ implies $d_2 - x_1 \geq 0$.

For $\ell\in [n]$, let $J_\ell$ denote the evaluation of the
$n+1-\ell$ inside integrals of (\ref{integrals}), that is, the
integrals with respect to the variables $x_n$, $x_{n-1}$, \ldots,
$x_{\ell}$. 

\begin{lemma}\label{inside_integrals}
For $i\in [k]$, we have
\begin{eqnarray}\label{int_lemma}
\nonumber
J_{\rho_i+1}
&=&
(-1)^{\delta_{i+1}+\delta_{i+3}+\delta_{i+5}+\cdots}
\\
&&
\ \ \ \ \ \ \ \cdot\sum_{\alpha\in\Kappa_{(0,\delta_i,\delta_{i+1},\ldots,\delta_{k-1})}}
(-1)^{\alpha_1+\alpha_3+\alpha_5+\cdots}
\cdot{1\over \alpha_1!\ \alpha_2!\ \cdots}\cdot 
x_{\rho_i}^{\alpha_1} d_{i+1}^{\alpha_2} d_{i+2}^{\alpha_3} \cdots. 
\end{eqnarray}
\end{lemma}

\begin{proof}
We prove the lemma by induction on $i$, starting with the trivial base
case of $i=k$, in which we have $J_{\rho_k+1} = J_{n+1} = 1$. Now suppose
the claim is true for some $i$. By straightforward iterated integration one
can show that for non-negative integers $r$ and $s$,
\begin{equation}\label{int_identity}
\int_0^a \int_0^{a-y_{r}} \int_0^{a-y_{r}-y_{r-1}} \dots
\int_0^{a-y_{r}-y_{r-1}-\cdots-y_2}
y_1^s\ dy_1\ \cdots\ d_{y_{r-1}}\ d_{y_{r}}
= {s!\ a^{r+s}\over (r+s)!}.
\end{equation}
Using (\ref{int_identity}) to 
integrate the term of
(\ref{int_lemma})
corresponding to a particular $\alpha \in
\Kappa_{(0,\delta_i,\delta_{i+1},\ldots,\delta_{k-1})}$, we get
\begin{eqnarray}\label{int_one_term}
\nonumber
&&
\int_0^{d_i-x_{\rho_{i-1}}}
\int_0^{d_i-x_{\rho_{i-1}}-x_{\rho_{i-1}+1}}
\dots
\int_0^{d_i-x_{\rho_{i-1}}-x_{\rho_{i-1}+1}-\cdots-x_{\rho_i-1}}\\
\nonumber
&&\ \ \ \ 
(-1)^{\alpha_1+\alpha_3+\cdots}
\cdot {1\over \alpha_1!\ \alpha_2!\ \cdots}\cdot
x_{\rho_i}^{\alpha_1} d_{i+1}^{\alpha_2} d_{i+2}^{\alpha_3} \cdots\ 
dx_{\rho_i}\ \cdots\ dx_{\rho_{i-1}+1}\\[6pt]
\nonumber
&&
= (-1)^{\alpha_1+\alpha_3+\cdots}
\cdot {1\over \alpha_1!\ \alpha_2!\ \cdots} \cdot
{\alpha_1!\ (d_i-x_{\rho_{i-1}})^{\delta_{i-1}+\alpha_1}\over
(\delta_{i-1}+\alpha_1)!} \cdot
d_{i+1}^{\alpha_2} d_{i+2}^{\alpha_3} \cdots\\[6pt]
\nonumber
&&
= (-1)^{\alpha_1+\alpha_3+\cdots}\cdot 
\sum_{j,m\geq 0\ :\atop j+m=\delta_{i-1}+\alpha_1}
{1\over \alpha_2!\ \alpha_3!\cdots}
\cdot
(-1)^j \cdot
{x_{\rho_{i-1}}^j d_i^{m}\over
 j!\ m!}\cdot
d_{i+1}^{\alpha_2} d_{i+2}^{\alpha_3} \cdots\\[6pt]
&&
= 
\sum_{{{j,m\ :\ j+m=\delta_{i-1}+\alpha_1,\atop
{\ \ \ \ \ \ (j,m,\alpha_2,\alpha_3,\ldots)\in\atop
\ \ \ \ \ \ \ \ \ \Kappa_{(0,\delta_{i-1},\delta_i,\ldots)}}}}}
(-1)^{j+\alpha_1+\alpha_3+\cdots}\cdot 
{1\over j!\ m!\ \alpha_2!\ \alpha_3!\ \cdots} \cdot
x_{\rho_{i-1}}^j d_i^m d_{i+1}^{\alpha_2} d_{i+2}^{\alpha_3} \cdots
\end{eqnarray}
Observe that $(j,m,\alpha_2,\alpha_3,\ldots)\in
\Kappa_{(0,\delta_{i-1},\delta_i,\ldots)}$ if and only if
$(\alpha_1,\alpha_2,\ldots)\in\Kappa_{(0,\delta_i,\delta_{i+1},\ldots)}$,
where $\alpha_1 = j+m-\delta_{i-1}$. Hence summing the above equation
over all $\alpha\in\Kappa_{(0,\delta_i,\delta_{i+1},\ldots)}$, we get
\begin{eqnarray*}\label{int_sum}
&&
J_{\rho_{i-1}+1} \\[6pt]
&&
=
\int_0^{d_i-x_{\rho_{i-1}}}
\int_0^{d_i-x_{\rho_{i-1}}-x_{\rho_{i-1}+1}}
\dots
\int_0^{d_i-x_{\rho_{i-1}}-x_{\rho_{i-1}+1}-\cdots-x_{\rho_i-1}}
J_{\rho_i+1}\ 
dx_{\rho_i}\ \cdots\ dx_{\rho_{i-1}+1}
\\[6pt]
&&
= (-1)^{\delta_{i}+\delta_{i+2}+\delta_{i+4}+\cdots}\\
&&
\ \ \ \ \cdot
\sum_{(j,m,\alpha_2,\alpha_3,\ldots)\atop
\ \ \ \ \in\Kappa_{(0,\delta_{i-1},\delta_i,\ldots)}}
(-1)^{j+\alpha_2+\alpha_4+\cdots}\cdot 
{1\over j!\ m!\ \alpha_2!\ \alpha_3!\ \cdots} \cdot
x_{\rho_{i-1}}^j d_i^m d_{i+1}^{\alpha_2} d_{i+2}^{\alpha_3} \cdots .
\end{eqnarray*}
Note that the signs are consistent: taking into account the factor
$(-1)^{\delta_{i+1}+\delta_{i+3}+\delta_{i+5}+\cdots}$ omitted from
(\ref{int_one_term}), the total sign of a term of
(\ref{int_one_term}) is
$$
(-1)^{\delta_{i+1}+\delta_{i+3}+\delta_{i+5}+\cdots}\cdot
(-1)^{j+\alpha_1+\alpha_3+\cdots}\\
=
(-1)^{\delta_{i}+\delta_{i+2}+\delta_{i+4}+\cdots}\cdot
(-1)^{j+\alpha_2+\alpha_4+\cdots},
$$
which is true because 
$$
\alpha_1+\alpha_2+\cdots = \delta_i+\delta_{i+1}+\cdots = n-\rho_i,
$$
and hence
all the exponents on both sides add up to
$2(n-\rho_{i})+2j$, i.e.\ an even number.
\end{proof}

To finish the proof of Theorem~\ref{vol_ZS_thm}, set $i=1$ in
Lemma~\ref{inside_integrals} and integrate with respect to $x_n$:
\begin{eqnarray*}
&&
n!\cdot J_n\\[6pt]
&&
= n!\int_0^{d_1} J_1\ dx_1\\[6pt]
&&
= (-1)^{\delta_1+\delta_3+\delta_5+\cdots}\ n!\\
&&\ \ \ \ 
\cdot
\sum_{\alpha\in\Kappa_{(0,\delta_1,\delta_2,\ldots)}}
(-1)^{\alpha_2+\alpha_4+\alpha_6+\cdots}\cdot
{1\over (\alpha_1+1)!\ \alpha_2!\ \alpha_3!\ \cdots}\cdot
d_1^{\alpha_1+1} d_2^{\alpha_2} d_3^{\alpha_3} \cdots \\[6pt]
&&=
(-1)^{\delta_1+\delta_3+\cdots}
\sum_{(\alpha_1+1,\alpha_2,\alpha_3,\ldots)\atop
\ \ \ \ \in\Kappa_{(1,\delta_1,\delta_2,\ldots)}}
(-1)^{\alpha_2+\alpha_4+\cdots}
{n\choose \alpha_1+1,\ \alpha_2,\ \alpha_3,\ldots}\cdot
d_1^{\alpha_1+1}d_2^{\alpha_2}d_3^{\alpha_3}\cdots,
\end{eqnarray*}
and it is clear that $(\alpha_1+1,\alpha_2,\alpha_3,\ldots)
\in \Kappa_{(1,\delta_1,\delta_2,\ldots)}$ if and only if
$\alpha \in \Kappa_{(0,\delta_1,\delta_2,\ldots)}$.
\end{proof}

\end{document}